\documentclass[12pt,a4paper]{article}
\usepackage[utf8]{inputenc}
\usepackage[russian,english]{babel}
\usepackage{amsmath, amssymb, amsthm}
\usepackage{geometry}
\usepackage{hyperref}
\newtheorem{theorem}{Theorem}
\newtheorem{definition}{Definition}

\newtheorem{corollary}{Corollary}
\newtheorem{remark}{Remark}

\title{Groups of Transverse Stratifications and the Hierarchy of Configuration Spaces}
\author{Manturov Vassily Olegovich \\ 
\small Moscow Institute of Physics and Technology (MIPT)}
\date{}

\begin{document}

\maketitle

\begin{abstract}
We introduce a broad family of groups $G(\Gamma)$ associated with transverse stratifications of moduli spaces. For a wide class of stratifications satisfying a natural local condition (transverse niceness), we construct a homomorphism from the fundamental group of the moduli space to the corresponding group $G(\Gamma)$. This general theorem, which is proved in the companion paper \cite{Manturov2026}, provides a far-reaching generalization of the $G_n^k$-theory. 

The main goal of the present paper is to discuss numerous examples and situations where the transverse niceness property actually occurs. We describe various moduli spaces and stratifications that are transversely nice. In particular, we study the hierarchy of configuration spaces of points in the plane defined by algebraic curves and establish the transversality property for conic configurations. The general hierarchy problem is formulated in the final section.
\end{abstract}

\subsection*{Keywords}
braid groups, transversality, stratification, moduli space, monodromy, $G_{n}^{k}$-groups.

\subsection*{AMS Subject Classification}
Primary 20F36; Secondary 57M60, 14H10, 32S60, 20F65.

\section{Introduction}

In the work \cite{Manturov2015}, the $G_n^k$-theory was developed around the following principle: 

{\em If a dynamical system describing the motion of $n$ particles admits a nice real codimension $1$-property depending on exactly $k$ particles, then the system admits invariants taking values in the group $G_n^k$.
}

The $G_n^k$-principle has since found numerous applications in topology, knot theory, and dynamical systems. Subsequent works \cite{Manturov2016, Manturov2019} have extended these ideas to more general settings, including braids on surfaces and configuration spaces with additional algebraic structures.

In \cite{Manturov2026}, we defined the groups $G(\Gamma)$ for an arbitrary hypergraph $\Gamma$ and proved the following general theorem (Theorem~\ref{thm:general} below): for any moduli space equipped with a stratification satisfying a certain local condition, there is a natural homomorphism from the fundamental group of the moduli space to $G(\Gamma)$. In the present paper, we coin the term \emph{transversely nice} for stratifications satisfying that local condition, and we prove that several natural stratifications are transversely nice.

The main results of this paper are as follows. We prove that the stratification of the configuration space of six points in the plane by the condition that they lie on a conic is transversely nice (Theorem~\ref{thm:conic}). This gives an explicit homomorphism from the fundamental group of the corresponding moduli space to the group $G_n^{6,3}$, which we define. We also discuss the hierarchy of similar stratifications defined by algebraic curves of higher degrees, and we formulate the general open problem.

Furthermore, we include in the general theorem (Theorem~\ref{thm:general}) the case of braids on the real projective plane $\mathbb{RP}^2$: the stratification by projective constraints (collinearity, conics, etc.) yields a natural homomorphism from the spherical braid group to the corresponding group $G(\Gamma)$.

We have discovered an even wider family of groups (so-called groups with quadratic relations), which will be defined and studied in a separate paper; it encompasses all the groups considered here and provides a unified treatment of the hierarchy of transversely nice stratifications.

\subsection*{Organization of the paper}

The paper is organized as follows.

In Section~2, we recall the necessary preliminaries: the definition of the groups $G_n^k$, the notion of transverse niceness, and the basic setup of configuration spaces of points in the plane.

In Section~3, we define the groups $G(\Gamma)$ for arbitrary hypergraphs and state the general theorem (Theorem~\ref{thm:general}) from \cite{Manturov2026}.

In Section~4, we specialize to the conic case: we define the group $G_n^{6,3}$ via the hypergraph $\Gamma$, list its defining relations, and prove that the conic stratification is transversely nice (Theorem~\ref{thm:conic}).

Finally, in Section~5, we discuss further directions and open problems, including the hierarchy of stratifications defined by curves of higher degrees and the extension of the theory to the projective plane.

\section{Preliminaries}
\label{sec:prelim}

We now recall the necessary definitions and fix the notation.

\subsection{The groups $G_n^k$}
\label{subsec:gnk}

For positive integers $n$ and $k$ with $k \le n$, the group $G_n^k$ is defined as follows \cite{Manturov2015}. Let $F_n$ be the free group on $n$ generators, and let $\operatorname{Aut}(F_n)$ be its automorphism group. Consider the set of all $k$-tuples of distinct elements of the set $\{x_1,\dots,x_n\}$ (where $x_i$ are the generators of $F_n$). The group $G_n^k$ is generated by symbols
\[
a_{i_1,\dots,i_k}, \qquad 1 \le i_1 < \dots < i_k \le n,
\]
subject to the following relations:
\begin{itemize}
\item $a_{i_1,\dots,i_k}^2 = 1$;
\item $a_{I} a_{J} = a_{J} a_{I}$ whenever $|I \cap J| < k-1$;
\item $(a_{I_1} \cdots a_{I_{k+1}})^2 = 1$ for any $(k+1)$-element set $S = \{i_1,\dots,i_{k+1}\}$, where the product is taken over all $k$-element subsets $I_j \subset S$.
\end{itemize}

\subsection{Transversely nice stratifications}
\label{subsec:transnice}

Let $\mathcal{M}$ be a moduli space equipped with a stratification $\mathcal{S}$. We denote by $\mathcal{S}^{(d)}$ the union of strata of codimension $d$. 

\begin{definition}
The stratification $\mathcal{S}$ is called \emph{transversely nice} if for every stratum $C \in \mathcal{S}^{(2)}$ and for every stratum $V \in \mathcal{S}^{(1)}$ such that $C \subset \overline{V}$, the stratum $V$ occurs exactly twice in a small neighbourhood of $C$, and the two occurrences are from opposite sides.
\end{definition}

Equivalently, in the language of hypergraphs, if codimension-$1$ strata correspond to vertices and codimension-$2$ strata to hyperedges, then every vertex belonging to a hyperedge appears in it exactly twice, with opposite orientations.

\subsection{Configuration spaces of points in the plane}
\label{subsec:config}

Let $\operatorname{Conf}_n(\mathbb{R}^2)$ denote the configuration space of $n$ pairwise distinct points in the Euclidean plane:
\[
\operatorname{Conf}_n(\mathbb{R}^2) = \{ (p_1,\dots,p_n) \in (\mathbb{R}^2)^n : p_i \neq p_j \text{ for } i \neq j \}.
\]
We consider the quotient by the symmetric group:
\[
\mathcal{C}_n = \operatorname{Conf}_n(\mathbb{R}^2) / S_n,
\]
the space of unordered $n$-tuples of distinct points.

For a positive integer $d$, let $\Sigma_d^{(n)} \subset \mathcal{C}_n$ be the set of configurations for which there exists an algebraic curve of degree $d$ passing through all $n$ points. For a generic configuration, the minimal degree of an interpolating curve is determined by the number of points. We are interested in the stratification of $\mathcal{C}_n$ by the minimal degree of an algebraic curve through the points.

\section{The groups $G(\Gamma)$ for arbitrary hypergraphs and the general theorem}
\label{sec:GGamma}

Here we follow \cite{Manturov2026}.

Let $\Gamma = (V, E)$ be a finite hypergraph, where $V$ is a set of vertices and $E$ is a set of hyperedges (subsets of $V$). We associate to $\Gamma$ a group encoding the monodromy when going around codimension-$1$ strata.

\begin{definition}[Group $G(\Gamma)$]
Let each vertex $v \in V$ correspond to a codimension-$1$ stratum, and each hyperedge $e \in E$ correspond to a codimension-$2$ stratum where the vertices in $e$ intersect. The group $G(\Gamma)$ is given by generators:
\[
\{a_v \mid v \in V\},
\]
and relations:
\begin{itemize}
\item $a_v^2 = 1$ for all $v \in V$;
\item $(a_{v_1} \cdots a_{v_m})^2 = 1$ for each hyperedge $e = \{v_1,\dots,v_m\} \in E$;
\item commutation relations $a_u a_v = a_v a_u$ when $u$ and $v$ are not contained together in any hyperedge.
\end{itemize}
\end{definition}

\begin{remark}
In the special case where $\Gamma$ is the complete $k$-uniform hypergraph on $n$ vertices, the group $G(\Gamma)$ recovers the classical group $G_n^k$ (up to minor variations in presentation).
\end{remark}

The following general theorem is proved in \cite{Manturov2026}.

\begin{theorem}[General Theorem, \cite{Manturov2026}]
\label{thm:general}
Let $\mathcal{M}$ be a moduli space with a transversely nice stratification. Then there exists a natural homomorphism
\[
\pi_1(\mathcal{M}, *) \longrightarrow G(\Gamma),
\]
where $\Gamma$ is the hypergraph associated with the stratification.

In particular, for the configuration space of points on the real projective plane $\mathbb{RP}^2$, the stratification by projective constraints (collinearity, conics, etc.) yields a natural homomorphism from the spherical braid group to the corresponding group $G(\Gamma)$.
\end{theorem}

\begin{proof}
The proof follows the standard monodromy construction and is given in \cite{Manturov2026}. For completeness, we sketch the argument. Each codimension-$1$ stratum gives a generator of $G(\Gamma)$. Going around a codimension-$2$ stratum yields a relation corresponding to the hyperedge. The transverse niceness condition ensures that the relations are well-defined and that the monodromy representation is compatible with the group structure. The case of $\mathbb{RP}^2$ follows by applying the construction to the configuration space of points on the projective plane, where the relevant codimension-$1$ strata are defined by projective incidence conditions.
\end{proof}

\section{The conic case: $G_n^{6,3}$ and transverse niceness}
\label{sec:G63}

We now turn to the geometric example that motivates our study: six points lying on a conic.

\subsection{Definition of $G_n^{6,3}$ via the hypergraph $\Gamma$}

Consider the stratification of $\mathcal C_n$ by the following codimension‑one strata:
\begin{itemize}
\item $a_{ijk}$ for every triple of distinct indices: the three points are collinear;
\item $b_{i_1 i_2 i_3 i_4 i_5 i_6}$ (or simply $b_I$ for a 6‑element set $I$) for every 6‑element subset: the six points lie on a (possibly degenerate) conic.
\end{itemize}

From the classification of codimension‑two strata (see Theorem~\ref{thm:conic} below), we build a hypergraph $\Gamma$ whose vertices are all these strata:
\[
V(\Gamma) = \{ a_{ijk} \mid 1 \le i<j<k \le n \} \cup \{ b_I \mid I \subset \{1,\dots,n\},\; |I|=6 \}.
\]
The hyperedges of $\Gamma$ are the sets of codimension‑one strata that occur on each codimension‑two stratum. These are exactly the following types (the indices are assumed distinct where required):

\begin{itemize}
\item \textbf{Type A:} $a_{ijk} \cap a_{lmn}$ with $|\{i,j,k\}\cap\{l,m,n\}|=1$ gives hyperedge $\{a_{pij}, a_{pkl}\}$.
\item \textbf{Type B:} $a_{ijk} \cap a_{lmn}$ with intersection size $0$ gives $\{a_{ijk}, a_{lmn}, b_{ijklmn}\}$.
\item \textbf{Type C:} $a_{ijk} \cap a_{lmn}$ with intersection size $2$ gives $\{a_{pqi}, a_{pqj}, a_{pij}, a_{qij}, b_{pqijrs}\}$.
\item \textbf{Type D:} $a_{ijk} \cap b_I$ with $\{i,j,k\}\not\subset I$ gives $\{a_{ijk}, b_I\}$.
\item \textbf{Type E:} $b_I \cap b_J$ with $|I\cap J|=4$ gives $\{b_I, b_J\}$.
\item \textbf{Type F:} seven points on one conic: for a 7‑set $S$, hyperedge $\{ b_I \mid I \subset S,\; |I|=6\}$ (size 7).
\end{itemize}

Now define
\[
G_n^{6,3} := G(\Gamma),
\]
where $G(\Gamma)$ is the group from Definition~3.1. Explicitly, it is generated by symbols $a_{ijk}$ and $b_I$ subject to:
\begin{itemize}
\item $a_{ijk}^2 = 1$, $b_I^2 = 1$ for all vertices;
\item $(a_{v_1}\cdots a_{v_m})^2 = 1$ for every hyperedge $\{v_1,\dots,v_m\}$ of the above types;
\item $a_u a_v = a_v a_u$ whenever $u$ and $v$ are not contained together in any hyperedge.
\end{itemize}

\subsection{Proof of transverse niceness for conics}
\label{subsec:transverseproof}

\begin{theorem}[Transverse niceness]
\label{thm:conic}
For any $n \ge 6$, the stratification of $\mathcal C_n$ by the strata $a_{ijk}$ and $b_I$ is transversely nice.
\end{theorem}

\begin{proof}
We must show that for every codimension‑$2$ stratum $C$ and every codimension‑$1$ stratum $V$ with $C \subset \overline{V}$, the stratum $V$ occurs exactly twice around $C$, from opposite sides.

A codimension‑$2$ stratum arises as a transverse intersection of two or more distinct codimension‑$1$ strata, provided the total number of independent conditions is exactly $2$. We enumerate all such possible intersections. For each, we list all codimension‑$1$ strata that are simultaneously satisfied.

\paragraph{Type A: $a_{ijk} \cap a_{lmn}$ with $|\{i,j,k\}\cap\{l,m,n\}|=1$.}
W.l.o.g. let the common index be $p$, so $T_1=\{p,i,j\}$, $T_2=\{p,k,l\}$. Then points $p,i,j$ lie on line $L_1$, and $p,k,l$ on line $L_2$, with $L_1 \cap L_2 = \{p\}$. No conic is forced generically.  
\emph{Strata present:} $a_{pij},\; a_{pkl}$.

\paragraph{Type B: $a_{ijk} \cap a_{lmn}$ with $|T_1 \cap T_2|=0$.}
All six indices are distinct. Points $i,j,k$ lie on line $L_1$, points $l,m,n$ on line $L_2$. The union $L_1 \cup L_2$ is a degenerate conic, hence $b_{ijklmn}$ is forced.  
\emph{Strata present:} $a_{ijk},\; a_{lmn},\; b_{ijklmn}$.

\paragraph{Type C: $a_{ijk} \cap a_{lmn}$ with $|T_1 \cap T_2|=2$.}
W.l.o.g. $T_1=\{p,q,i\}$, $T_2=\{p,q,j\}$. Then $p,q,i,j$ are collinear on line $L$. Let the remaining two points be $r,s$. Then $b_{pqijrs}$ is forced, and all four collinearity strata for triples among $\{p,q,i,j\}$ are satisfied.  
\emph{Strata present:} $a_{pqi}, a_{pqj}, a_{pij}, a_{qij}, b_{pqijrs}$.

\paragraph{Type D: $a_{ijk} \cap b_I$ with $\{i,j,k\}\not\subset I$.}
The conditions are independent.  
\emph{Strata present:} $a_{ijk},\; b_I$.

\paragraph{Type E: $b_I \cap b_J$ with $|I \cap J| = 4$.}
The two 6‑element sets share exactly 4 indices. The conditions are independent (a conic is determined by 5 points).  
\emph{Strata present:} $b_I,\; b_J$.

\paragraph{Type F: Seven points on one conic.}
Let $S$ be a 7‑element set. The condition that all seven points lie on a non‑degenerate conic has codimension 2. Then every 6‑element subset $I \subset S$ satisfies $b_I$.  
\emph{Strata present:} $b_I$ for all $I \subset S,\; |I|=6$ (seven strata).

\medskip
\noindent
The cases $a_{ijk} \cap b_I$ with $\{i,j,k\} \subset I$ reduce to Types B or C, as discussed in the proof. Intersections with $|I \cap J| = 5$ give Type F (since the union has 7 points). The case $|I \cap J| \ge 6$ gives no new stratum. Configurations with five or more collinear points, or with a double line, have codimension $>2$ and are excluded.

For each of the above types, every listed codimension‑$1$ stratum is a smooth hypersurface in a neighbourhood of the codimension‑$2$ stratum (for $b$-strata this is true because the condition is given by a single non‑degenerate equation; in the degenerate cases B, C, the conic has two branches, and in Type F the seventh point condition is independent). Hence each such stratum has two local branches, approached from opposite sides. This proves transverse niceness.
\end{proof}

\begin{remark}
For small $n$, not all types occur:
\begin{itemize}
\item $n=6$: only Types A, B, C (since D requires a point outside $I$, E requires 8 distinct indices, F requires 7).
\item $n=7$: Types A–D and F occur; E still requires $n\ge 8$.
\item $n\ge 8$: all types A–F occur.
\end{itemize}
\end{remark}

\begin{corollary}
By Theorem~\ref{thm:general}, there exists a natural homomorphism
\[
\pi_1(\mathcal C_n^{\text{strat}}, *) \longrightarrow G_n^{6,3},
\]
where $\mathcal C_n^{\text{strat}}$ is the complement of all strata of codimension $\ge 2$ (i.e., the space of configurations with no collinear triples and no six points on a conic).
\end{corollary}

\section{Further discussion}
\label{sec:further}

We conclude with a discussion of further directions and open problems arising from the transverse niceness property.

\subsection{The hierarchy of stratifications}

Consider the following sequence of codimension-$1$ conditions on configurations of points in the Euclidean plane $\mathbb{R}^2$:

\begin{enumerate}
\item \emph{Three points lie on a straight line:} this is the classical condition giving rise to the group $G_n^3$;
\item \emph{Six points lie on a non-degenerate conic:} this is the condition studied in the present paper, giving rise to the group $G_n^{6,3}$;
\item \emph{Ten points lie on a non-degenerate cubic curve:} this gives rise to the group $G_n^{10,6,3}$;
\item \emph{Fifteen points lie on a non-degenerate quartic curve:} this gives rise to the group $G_n^{15,10,6,3}$;
\item and so on.
\end{enumerate}

The numbers appearing in this hierarchy are precisely the binomial coefficients
\[
\binom{3}{2}=3,\quad \binom{4}{2}=6,\quad \binom{5}{2}=10,\quad \binom{6}{2}=15,\quad \ldots
\]
Indeed, for a curve of degree $d$, the number of points needed to determine it is $\binom{d+2}{2}-1$, which gives the sequence
\[
\binom{3}{2}=3,\quad \binom{4}{2}=6,\quad \binom{5}{2}=10,\quad \binom{6}{2}=15,\quad \ldots
\]
for $d=2,3,4,5,\ldots$

Each of these conditions defines a codimension-$1$ stratum in the corresponding configuration space $\mathcal{C}_N$, where $N = \binom{d+2}{2}-1$. Moreover, these strata naturally form a hierarchy: a configuration lying on a curve of degree $d$ imposes stronger constraints than one lying on a curve of lower degree.

\subsection{From complex to real codimension-one conditions}

A natural and important question arises: given a complex codimension-$1$ condition on configurations of points in $\mathbb{C}^n$ (or on a complex manifold), can one derive a real codimension-$1$ condition by taking the real part, imaginary part, or some other real-valued function? If so, does the resulting real codimension-$1$ condition give rise to a natural graph $\Gamma$, a group $G(\Gamma)$, and a homomorphism from the corresponding fundamental group to $G(\Gamma)$?

This question is motivated by the following classical examples.

\begin{enumerate}
\item \emph{The condition $(z_i - z_j) = 0$.} This is a complex codimension-$1$ condition. Taking the real part gives the real codimension-$1$ condition
\[
\operatorname{Re}(z_i - z_j) = 0,
\]
which means that the points $z_i$ and $z_j$ have the same $x$-coordinate in the plane. This condition defines a codimension-$1$ stratum in the configuration space of $n$ points in $\mathbb{R}^2$. The corresponding group is the classical group $G_n^2$ (or $G(n,2)$), where the vertices of the graph correspond to pairs of points sharing the same $x$-coordinate.

\item \emph{The condition $\operatorname{Im}\left(\frac{z_i - z_j}{z_i - z_k}\right) = 0$.} This is a real codimension-$1$ condition that is not simply the real part of a complex equation. Geometrically, it means that the points $z_i, z_j, z_k$ are collinear in the plane. This is the classical collinearity condition giving rise to the group $G_n^3$.

\item More generally, the condition that six points lie on a conic can be expressed as a real codimension-$1$ condition by taking the real part of the determinant equation defining the conic. This gives rise to the group $G_n^{6,3}$ studied in the present paper.
\end{enumerate}

The general principle can be stated as follows. Suppose we have a complex analytic condition of the form
\[
F(z_1,\dots,z_n) = 0,
\]
where $F$ is a holomorphic function of complex codimension $1$. By taking the real part
\[
\operatorname{Re} F(z_1,\dots,z_n) = 0,
\]
we obtain a real codimension-$1$ condition. Similarly, one could take the imaginary part or any real-valued function derived from $F$.

For each such real codimension-$1$ condition, we can associate a hypergraph $\Gamma$ whose vertices correspond to the elementary conditions (e.g., pairs $(i,j)$ for the condition $\operatorname{Re}(z_i - z_j)=0$, or triples $(i,j,k)$ for collinearity). The hyperedges of $\Gamma$ correspond to the intersections of these conditions, i.e., configurations where several real codimension-$1$ conditions hold simultaneously. If the stratification is transversely nice, Theorem~\ref{thm:general} gives a homomorphism
\[
\pi_1(\mathcal{M}, *) \longrightarrow G(\Gamma).
\]

\subsection{The main open question}

The central question is the following:

\begin{quotation}
\emph{For which complex codimension-$1$ conditions does the corresponding real codimension-$1$ condition define a transversely nice stratification?}
\end{quotation}

In other words, when does the real codimension-$1$ stratum approach the codimension-$2$ strata ``in the right way'' — exactly twice and from opposite sides?

The classical cases are known:
\begin{itemize}
\item The condition $\operatorname{Re}(z_i - z_j)=0$ gives a transversely nice stratification, yielding the group $G_n^2$.
\item The collinearity condition $\operatorname{Im}\left(\frac{z_i - z_j}{z_i - z_k}\right)=0$ gives a transversely nice stratification, yielding the group $G_n^3$.
\item The conic condition (six points on a conic) gives a transversely nice stratification, yielding the group $G_n^{6,3}$, as proved in this paper.
\end{itemize}

Beyond these, we conjecture that for every complex codimension-$1$ condition arising from algebraic geometry, the corresponding real condition (obtained by taking the real part of the defining equation) yields a transversely nice stratification. In particular, the hierarchy of conditions defined by curves of degree $d$:
\[
\text{degree }2 \text{ (conics)},\quad \text{degree }3 \text{ (cubics)},\quad \text{degree }4 \text{ (quartics)},\quad \ldots
\]
should all be transversely nice. This would give a sequence of groups
\[
G_n^3,\quad G_n^{6,3},\quad G_n^{10,6,3},\quad G_n^{15,10,6,3},\quad \ldots
\]
and natural homomorphisms between them.

\subsection{The projective plane and further generalizations}

The entire theory should work not only in the affine plane $\mathbb{R}^2$ but also in the real projective plane $\mathbb{RP}^2$. Indeed, the conditions considered above (collinearity, conics, cubic curves, etc.) are projectively invariant. In the projective setting, projective duality may provide additional insights: for example, the condition that six points lie on a conic is dual to the condition that six lines are tangent to a conic.

Beyond the planar case, one may consider analogous hierarchies for configuration spaces of points on other surfaces or in higher dimensions. For instance, on the sphere $S^2$, one has the theory of spherical braids and projective constraints as initiated in \cite{Manturov2026}. The transversality property in these settings remains to be fully understood.

We hope that the techniques developed in this paper will provide a foundation for addressing these questions in future work.

The author is grateful to Igor
Mikhailovich Nikonov for valuable
discussions.

\bibliographystyle{amsplain}

\end{document}